\documentclass{amsart}[12pt]
\copyrightinfo{2006}{American Mathematical Society}
\usepackage{mathrsfs,amscd}

\newtheorem{theorem}{Theorem}[section]
\newtheorem{lemma}[theorem]{Lemma}

\theoremstyle{definition}
\newtheorem{definition}[theorem]{Definition}
\newtheorem{example}[theorem]{Example}

\theoremstyle{remark}
\newtheorem{remark}[theorem]{Remark}

\numberwithin{equation}{section}

\begin{document}

\title[]{The rate of convergence of new Lax-Oleinik type operators for time-periodic positive definite Lagrangian systems}

\author[K. Wang and J. Yan]{Kaizhi Wang$^{\dagger,\,\ddagger}$ and Jun Yan$^{\ddagger}$}

\address{$^{\dagger}$ School of Mathematics,
Jilin University,  Changchun 130012, China}

\address{$^{\ddagger}$ School of Mathematical Sciences, Fudan
University, Shanghai 200433, China}

\email{yanjun@fudan.edu.cn}

\date{September 2011}

\begin{abstract}
Assume that the Aubry set of the time-periodic positive definite
Lagrangian $L$ consists of one hyperbolic 1-periodic orbit.
We provide an upper bound estimate of the rate of convergence of
the family of new Lax-Oleinik type operators associated with
$L$ introduced by the authors in \cite{W-Y}. In addition, we construct
an example where the Aubry set of a time-independent positive definite
Lagrangian system consists of one hyperbolic periodic orbit and
the rate of convergence of the Lax-Oleinik semigroup cannot be better
than $O(\frac{1}{t})$.
\end{abstract}

\maketitle

\section{Introduction}
In an earlier paper \cite{W-Y} the present authors introduced a
new kind of Lax-Oleinik type operator with parameters (hereinafter
referred to as new L-O operator) associated with time-periodic
positive definite Lagrangian systems in the context of the weak KAM theory,
and proved that the family of new L-O operators with an arbitrary continuous
function as initial condition converges to a backward weak KAM solution of the
corresponding Hamilton-Jacobi equation. In this paper we study the
rate of convergence of the family of new L-O operators under
the assumption that the Aubry set of the time-periodic positive
definite Lagrangian system consists of one hyperbolic
1-periodic orbit.

Let $M$ be a closed and connected smooth manifold of dimension
$m$. Denote by $TM$ its tangent bundle and $T^*M$ the cotangent
one. We choose, once and for all, a $C^\infty$ Riemannian metric
on $M$. It is classical that there is a canonical way to associate
to it a Riemannian metric on $TM$.  Consider a $C^\infty$
Lagrangian $L: TM\times\mathbb{R}^1\to\mathbb{R}^1$,
$(x,v,t)\mapsto L(x,v,t)$. We suppose that $L$ satisfies the
following conditions introduced by Mather \cite{Mat91}:

\vskip0.2cm
\begin{itemize}
    \item [(H1)] \textbf{Periodicity}. $L$ is 1-periodic in the
                 $\mathbb{R}^1$ factor, i.e.,
                 $L(x,v,t)=L(x,v,t+1)$ for all $(x,v,t)\in TM\times\mathbb{R}^1$.
    \item [(H2)] \textbf{Positive Definiteness}. For each $x\in M$ and each
                 $t\in\mathbb{R}^1$, the restriction of $L$ to $T_xM\times
                 t$ is strictly convex in the sense that its
                 Hessian second derivative is everywhere positive
                 definite.
    \item [(H3)] \textbf{Superlinear Growth}.
                 $\lim_{\|v\|_x\to+\infty}\frac{L(x,v,t)}{\|v\|_x}=+\infty$
                 uniformly on $x\in M$, $t\in\mathbb{R}^1$,
                 where $\|\cdot\|_x$ denotes the norm on $T_xM$ induced by
                 the Riemannian metric on $M$.
    \item [(H4)] \textbf{Completeness of the Euler-Lagrange Flow}.
                 The maximal solutions of the Euler-Lagrange
                 equation, which in local coordinates is:
                 \[
                 \frac{d}{dt}\frac{\partial L}{\partial
                 v}(x,\dot{x},t)=\frac{\partial L}{\partial
                 x}(x,\dot{x},t),
                 \]
                 are defined on all of $\mathbb{R}^1$.
\end{itemize}
\vskip0.2cm

The Euler-Lagrange equation is a second order periodic
differential equation on $M$ and generates a flow of
diffeomorphisms $\phi^L_t:TM\times\mathbb{S}^1\to
TM\times\mathbb{S}^1$, $t\in\mathbb{R}^1$, where $\mathbb{S}^1$
denotes the circle $\mathbb{R}^1/\mathbb{Z}$, defined by

\[
\phi^L_t(x_0,v_0,t_0)=(x(t+t_0),\dot{x}(t+t_0),(t+t_0)\
\mathrm{mod}\ 1),
\]
where $x:\mathbb{R}^1\to M$ is the maximal solution of the
Euler-Lagrange equation with initial conditions $x(t_0)=x_0$,
$\dot{x}(t_0)=v_0$. The completeness and periodicity conditions
grant that this correctly defines a flow on
$TM\times\mathbb{S}^1$. We can associate with $L$ a Hamiltonian,
as a function on $T^*M\times\mathbb{R}^1$: $H(x,p,t)=\sup_{v\in
T_xM}\{\langle p,v\rangle_x-L(x,v,t)\}$, where $\langle
\cdot,\cdot\rangle_x$ represents the canonical pairing between the
tangent and cotangent space. The corresponding Hamilton-Jacobi
equation is

\begin{align*}
w_t+H(x,w_x,t)=c(L),
\end{align*}
where $c(L)$ is the Ma$\mathrm{\tilde{n}}\mathrm{\acute{e}}$
critical value \cite{Man97} of the Lagrangian $L$. In terms of
Mather's $\alpha$ function $c(L)=\alpha(0)$. Without loss of
generality, we will from now on always assume $c(L)=0$.

Let us first recall the definition of the Lax-Oleinik semigroup
(hereinafter referred to as L-O semigroup) associated with $L$.
The L-O semigroup is well known in several domains, such as PDE,
Optimization and Control Theory, Calculus of Variations and
Dynamical Systems (especially in the Weak KAM Theory
\cite{Fat-b}). For each $t\geq 0$ and each $u\in
C(M,\mathbb{R}^1)$, let

\begin{align*}
T_tu(x)=\inf_\gamma\Big\{u(\gamma(0))+\int_0^tL(\gamma(s),\dot{\gamma}(s),s)ds\Big\}
\end{align*}
for all $x\in M$, where the infimum is taken among the continuous
and piecewise $C^1$ paths $\gamma:[0,t]\to M$ with $\gamma(t)=x$.
For each $t\geq 0$, $T_t$ is an operator from $C(M,\mathbb{R}^1)$
to itself. Since $L$ is time-periodic, then
$\{T_n\}_{n\in\mathbb{N}}$ is a one-parameter semigroup of
operators, called the L-O semigroup associated with $L$, where
$\mathbb{N}=\{0,1,2,\cdots\}$.

Fathi proved \cite{Fat4} the convergence of the full L-O semigroup
(i.e., $\{T^a_t\}_{t\geq 0}$) in the time-independent
case\footnote{The L-O semigroup associated with a time-independent
Lagrangian $L_a$ is the semigroup of operators $\{T^a_t\}_{t\geq
0}: C(M,\mathbb{R}^1)\to C(M,\mathbb{R}^1)$ defined by
\begin{align*}
T^a_tu(x)=\inf_\gamma\Big\{u(\gamma(0))+\int_0^tL_a(\gamma(s),\dot{\gamma}(s))ds\Big\},
\end{align*} where the infimum is taken among the continuous
and piecewise $C^1$ paths $\gamma:[0,t]\to M$ with
$\gamma(t)=x$.}. More precisely, he showed that for each $C^2$
superlinear and strictly convex Lagrangian $L_a:TM\to\mathbb{R}^1$
and each $u\in C(M,\mathbb{R}^1)$, the uniform limit, for
$t\to+\infty$, of $T^a_tu+c(L_a)t$ exists and the limit $\bar{u}$
is a backward weak KAM solution of the corresponding
Hamilton-Jacobi equation. In the same paper Fathi raised the
question as to whether the analogous result holds in the
time-periodic case. This would be the convergence of $T_nu+nc(L)$,
$\forall u\in C(M,\mathbb{R}^1)$, as $n\to+\infty$,
$n\in\mathbb{N}$. In view of the relation between $T_n$ and the
Peierls barrier $h$ (see \cite{Mat93} or \cite{Fat5,Ber,Con}), if
the liminf in the definition of the Peierls barrier is not a
limit, then the L-O semigroup in the time-periodic case does not
converge. Fathi and Mather \cite{Fat5} constructed examples where
the liminf in the definition of the Peierls barrier is not a
limit, thus answering the above question negatively.

As mentioned at the beginning of this paper, the authors
\cite{W-Y} introduced a new kind of operator with parameters
associated with $L$, called the new L-O operator, and proved the
convergence of the family of new L-O operators. Let us now
recall the definition of the new L-O operator and some important
results in \cite{W-Y}.

\begin{definition}[new L-O operator] For each $\tau\in[0,1]$, each $n\in\mathbb{N}$ and each
$u\in C(M,\mathbb{R}^1)$, let
\[
\tilde{T}_n^\tau u(x)= \inf_{k\in\mathbb{N} \atop n\leq k\leq
2n}\inf_{\gamma}\Big\{u(\gamma(0))+\int_{0}^{\tau+k}
L(\gamma(s),\dot{\gamma}(s),s)ds\Big\}
\]
for all $x\in M$, where the second infimum is taken among the
continuous and piecewise $C^1$ paths $\gamma:[0,\tau+k]\rightarrow
M$ with $\gamma(\tau+k)=x$.
\end{definition}

For each $\tau\in[0,1]$ and each $n\in \mathbb{N}$,
$\tilde{T}_n^\tau$ is an operator from $C(M,\mathbb{R}^1)$ to
itself. For more properties of the new L-O operator
$\tilde{T}_n^\tau$, we refer the reader to \cite{W-Y}. For each
$n\in\mathbb{N}$ and each $u\in C(M,\mathbb{R}^1)$, let
$U^u_n(x,\tau)=\tilde{T}_n^\tau u(x)$ for all $(x,\tau)\in
M\times[0,1]$. Then $U^u_n$ is a continuous function on
$M\times[0,1]$.

The main result of \cite{W-Y} is the following theorem.

\begin{theorem}\label{oth}
For each $u\in C(M,\mathbb{R}^1)$, the uniform limit
$\lim_{n\to+\infty}U^u_n$ exists and

\[
\lim_{n\to+\infty}U^u_n(x,\tau)=\inf_{y\in
M}\big(u(y)+h_{0,[\tau]}(y,x)\big)
\]
for all $(x,\tau)\in M\times[0,1]$, where $[\tau]=\tau$ mod 1, and
$h$ denotes the (extended) Peierls barrier. Furthermore, let
$\bar{u}(x,[\tau])=\inf_{y\in M}\big(u(y)+h_{0,[\tau]}(y,x)\big)$.
Then $\bar{u}:M\times\mathbb{S}^1\to\mathbb{R}^1$ is a backward
weak KAM solution of the Hamilton-Jacobi equation

\begin{align}\label{1-1}
w_s+H(x,w_x,s)=0.
\end{align}
\end{theorem}

Another important result of \cite{W-Y} states as follows.

\begin{theorem}
Let $\bar{u}\in C(M\times\mathbb{S}^1,\mathbb{R}^1)$. Then the
following three statements are equivalent.
\begin{itemize}
    \item There exists $u\in C(M,\mathbb{R}^1)$ such that the uniform limit
           $\lim_{n\to+\infty}U^u_n=\bar{u}$.
    \item $\bar{u}$ is a backward weak KAM solution of
          (\ref{1-1}).
    \item $\bar{u}$ is a viscosity solution of (\ref{1-1}).
\end{itemize}
\end{theorem}

The aim of the present paper is to derive the rate of convergence
of $U^u_n$ in a special case. More precisely, we will provide an
upper bound estimate of the rate of convergence of $U^u_n$
associated with a $C^\infty$ Lagrangian $L$, which satisfies
(H1)-(H4) and the following hypothesis.

\vskip0.1cm

(H5) The Aubry set consists of one hyperbolic 1-periodic
orbit.

\vskip0.1cm

Now we come to the major result of this paper.

\begin{theorem}\label{th}
If a $C^\infty$ Lagrangian $L:
TM\times\mathbb{R}^1\to\mathbb{R}^1$ satisfies the hypotheses
(H1)-(H5), then there exists $\rho>0$ such that for each $u\in
C(M,\mathbb{R}^1)$ there is $K>0$ such that
\[
\|U^u_n(x,\tau)-\bar{u}(x,[\tau])\|_\infty\leq Ke^{-\rho n}, \quad
\forall n\in\mathbb{N},
\]
where $\|\cdot\|_\infty$ denotes the supremum norm in the space
$C(M\times[0,1],\mathbb{R}^1)$.
\end{theorem}

We believe that
there is a deep relation between dynamical properties of the Aubry
set (Mather set) and the rates of convergence of the L-O semigroup
(time-independent case) and the family of new L-O operators.
We now would like to detail on available relative works in the
literature. All these results are for time-independent Lagrangian
systems.

\vskip0.1cm

\noindent I. {\em Results on the rate of convergence of the L-O
semigroup $\{T^a_t\}_{t\geq 0}$}:

\vskip0.1cm

In \cite{Itu}, Iturriaga and S$\mathrm{\acute{a}}$nchez-Morgado
proved that if the Aubry set consists in a finite number of
hyperbolic fixed points, the L-O semigroup converges
exponentially. At the end of this paper, we will construct an
example (Example \ref{ex}) to show that the rate of convergence of
the L-O semigroup cannot be better than $O(\frac{1}{t})$ under the
assumption that the Aubry set consists of a finite number of
hyperbolic periodic orbits.

The authors \cite{Wan} dealt with the
rate of convergence problem when the Mather set consists of
degenerate fixed points. More precisely, consider the Lagrangian
$L^0_a(x,v)=\frac{1}{2}v^2+V(x)$, $x\in \mathbb{S}^1$, $v\in
\mathbb{R}^1$, where $V$ is a real analytic function on
$\mathbb{S}^1$ and has a unique global minimum point $x_0$.
Without loss of generality, one may assume $x_0=0$, $V(0)=0$. Then
$c(L^0_a)=0$ and $\tilde{\mathcal{M}}_0=\{(0,0)\}$, where
$\tilde{\mathcal{M}}_0$ is the Mather set with cohomology class 0.
An upper bound estimate of the rate of convergence of the L-O
semigroup is provided under the assumption that $\{(0,0)\}$ is a
degenerate fixed point: for every $u\in
C(\mathbb{S}^1,\mathbb{R}^1)$, there exists a constant $K_1>0$
such that

\[
\|T^a_tu-\bar{u}\|_\infty\leq\frac{K_1}{\sqrt[k-1]{t}}, \quad
\forall t>0,
\]
where $k\in \mathbb{N}$, $k\geq 2$ depends only on the degree of
degeneracy of the minimum point of the potential function $V$.

In \cite{W-Y} the authors discussed the rate of convergence problem
when the Aubry set is a quasi-periodic invariant torus of the
Euler-Lagrange flow. Consider a class of $C^2$  superlinear and
strictly convex Lagrangians on $\mathbb{T}^n$

\begin{align}\label{1-2}
L^1_a(x,v)=\frac{1}{2}\langle
A(x)(v-\omega),(v-\omega)\rangle+f(x,v-\omega), \quad x\in
\mathbb{T}^n,\ v\in\mathbb{R}^n,
\end{align}
where $A(x)$ is an $n\times n$ matrix, $\omega\in\mathbb{S}^{n-1}$
is a given vector, and $f(x,v-\omega)=O(\|v-\omega\|^3)$ as
$v-\omega\rightarrow 0$. It is clear that $c(L^1_a)=0$ and
$\tilde{\mathcal{M}}_0=\tilde{\mathcal{A}}_0=\tilde{\mathcal{N}}_0=\cup_{x\in\mathbb{T}^n}(x,\omega)$,
which is a quasi-periodic invariant torus with frequency vector
$\omega$ of the Euler-Lagrange flow associated to $L^1_a$, where
$\tilde{\mathcal{A}}_0$ and $\tilde{\mathcal{N}}_0$ are the Aubry
set and the Ma$\mathrm{\tilde{n}}\mathrm{\acute{e}}$ set with
cohomology class 0, respectively. For (\ref{1-2}), the authors
showed that for each $u\in
C(\mathbb{T}^n,\mathbb{R}^1)$, there is a constant $K_2>0$ such
that

\begin{align}\label{1-3}
\|T^a_tu-\bar{u}\|_\infty\leq\frac{K_2}{t}, \quad \forall t>0.
\end{align}
An example was also provided in \cite{W-Y} to show that the above
result is sharp in the sense of order.

\vskip0.1cm

\noindent II. {\em Results on the rate of convergence of the new
L-O semigroup $\{\tilde{T}^a_t\}_{t\geq 0}$}\footnote{In
\cite{W-Y}, the authors also introduced a new kind of Lax-Oleinik
type operator $\tilde{T}^a_t$ associated with time-independent
Lagrangians. The new L-O semigroup associated with a
time-independent Lagrangian $L_a$ is the semigroup of operators
$\{\tilde{T}^a_t\}_{t\geq 0}: C(M,\mathbb{R}^1)\to
C(M,\mathbb{R}^1)$ defined by \[ \tilde{T}^a_tu(x)=\inf_{t\leq
\sigma\leq 2t}\inf_{\gamma}\Big\{u(\gamma(0))+\int_0^\sigma
L_a(\gamma(s),\dot{\gamma}(s))ds\Big\},
\]
where the second infimum is taken among the continuous and
piecewise $C^1$ paths $\gamma:[0,\sigma]\rightarrow M$ with
$\gamma(\sigma)=x$.}:

\vskip0.1cm

The authors showed in \cite{W-Y} that for each $C^2$ superlinear
and strictly convex Lagrangian $L_a$ with $c(L_a)=0$ and each
$u\in C(M,\mathbb{R}^1)$, the uniform limit
$\lim_{t\rightarrow+\infty}\tilde{T}^a_tu$ exists and $
\lim_{t\rightarrow+\infty}\tilde{T}^a_tu=\lim_{t\rightarrow+\infty}T^a_tu=\bar{u}.
$ Furthermore,
$\|\tilde{T}^a_tu-\bar{u}\|_\infty\leq\|T^a_tu-\bar{u}\|_\infty$,
$\forall t\geq 0$. It means that the new L-O semigroup converges faster than the
L-O semigroup. However, it does not necessarily mean that the new L-O semigroup converges
faster than the L-O semigroup in the sense of order.

Recall the notations for Diophantine vectors:
for $\varrho>n-1$ and $\alpha>0$, let

\[
\mathcal{D}(\varrho,\alpha)=\Big\{\beta\in \mathbb{S}^{n-1}\ |\
|\langle\beta,k\rangle|\geq\frac{\alpha}{|k|^\varrho},\ \forall
k\in\mathbb{Z}^n\backslash\{0\}\Big\},
\]
where $|k|=\sum_{i=1}^n|k_i|$. For (\ref{1-2}), the authors
\cite{W-Y} proved that given any frequency vector
$\omega\in\mathcal{D}(\varrho,\alpha)$, for each $u\in
C(\mathbb{T}^n,\mathbb{R}^1)$, there is a constant $K_3>0$ such
that

\begin{align}\label{1-4}
\|\tilde{T}^a_tu-\bar{u}\|_\infty\leq
K_3t^{-(1+\frac{4}{2\varrho+n})}, \quad \forall t>0.
\end{align}
In view of (\ref{1-3}) and (\ref{1-4}), we conclude that the new
L-O semigroup converges faster than the L-O semigroup in the sense
of order when the Aubry set $\tilde{\mathcal{A}}_0$ of the
Lagrangian system (\ref{1-2}) is a quasi-periodic invariant torus
with Diophantine frequency vector
$\omega\in\mathcal{D}(\varrho,\alpha)$.

The rest of the paper is organized as follows. Section 2 includes
some basic definitions and preliminary results. In Section 3 we
give the proof of Theorem \ref{th}.  Section 4 presents an example
(Example \ref{ex}) where the Aubry set of a time-independent positive
definite Lagrangian system consists of one
hyperbolic periodic orbit and the rate of convergence of the L-O
semigroup cannot be better than $O(\frac{1}{t})$.

\section{Preliminaries}
In this section we introduce the notations used in the sequel and
review some definitions and results of Mather and weak KAM
theories that we are going to use. In addition, we also prove two
preliminary lemmas.

In this paper, as is usual $\mathbb{S}^1=\mathbb{R}^1/\mathbb{Z}$,
whose universal cover is the Euclidean space $\mathbb{R}^1$. We
view $\mathbb{S}^1$ as a fundamental domain in $\mathbb{R}^1:
\bar{I}=[0,1]$ with the two endpoints identified. The unique
coordinate $s$ of a point in $\mathbb{S}^1$ will belong to
$I=[0,1)$. The standard universal covering projection
$\pi:\mathbb{R}^1\to\mathbb{S}^1$ takes the form
$\pi(\tilde{s})=[\tilde{s}]$, where $[\tilde{s}]=\tilde{s}$ mod 1,
denotes the fractional part of $\tilde{s}$
($\tilde{s}=\{\tilde{s}\}+[\tilde{s}]$, where $\{\tilde{s}\}$ is
the greatest integer not greater than $\tilde{s}$). $\|\cdot\|$
denotes the usual Euclidean norm.

As done by Mather in \cite{Mat93}, it is convenient to introduce,
for all $t'\geq t$ and $x$, $y\in M$, the following quantity:

\[
F_{t,t'}(x,y)=\inf_\gamma\int_t^{t'}L(\gamma(s),\dot{\gamma}(s),s)ds,
\]
where the infimum is taken over the continuous and piecewise $C^1$
paths $\gamma:[t,t']\to M$ such that $\gamma(t)=x$ and
$\gamma(t')=y$.

Following Ma$\mathrm{\tilde{n}}\mathrm{\acute{e}}$ \cite{Man97}
and Mather \cite{Mat93}, define the action potential and the
extended Peierls barrier as follows.

\vskip0.1cm

\noindent {\em Action Potential}: for each
$(s,s')\in\mathbb{S}^1\times\mathbb{S}^1$, let

\[
\Phi_{s,s'}(x,x')=\inf F_{t,t'}(x,x')
\]
for all $(x,x')\in M\times M$, where the infimum is taken on the
set of $(t,t')\in\mathbb{R}^2$ such that $s=[t]$, $s'=[t']$ and
$t'\geq t+1$.

\vskip0.1cm

\noindent {\em Extended Peierls Barrier}: for each
$(s,s')\in\mathbb{S}^1\times\mathbb{S}^1$, let

\begin{align*}
 h_{s,s'}(x,x')=\liminf_{t'-t\to+\infty}F_{t,t'}(x,x')
\end{align*}
for all $(x,x')\in M\times M$, where the liminf is restricted to
the set of $(t,t')\in\mathbb{R}^2$ such that $s=[t]$, $s'=[t']$.
It can be shown that the extended Peierls barrier $h$ is Lipschitz
(see \cite{Con}).

\vskip0.1cm

A continuous and piecewise $C^1$ curve $\gamma:\mathbb{R}^1\to M$
is called global semi-static if

\[
\int_t^{t'}L(\gamma,\dot{\gamma},s)ds=\Phi_{[t],[t']}(\gamma(t),\gamma(t'))
\]
for all $[t,t']\subset\mathbb{R}^1$. An orbit
$(\gamma(s),\dot{\gamma}(s),[s])$ is called global semi-static if
$\gamma$ is a global semi-static curve. The
Ma$\mathrm{\tilde{n}}\mathrm{\acute{e}}$ set
$\tilde{\mathcal{N}}_0$ is the union in $TM\times\mathbb{S}^1$ of
the images of global semi-static orbits.

For each $n\in\mathbb{N}$ and each $(\tau,\tau',x,x')\in
[0,1]\times[0,1]\times M\times M$, let

\[
\mathcal{F}_n(\tau,\tau',x,x')=\inf_{k\in\mathbb{N} \atop n\leq
k\leq 2n}F_{\tau,\tau'+k}(x,x').
\]
Then from Proposition 3.5 in \cite{W-Y},

\begin{align}\label{2-1}
\lim_{n\to+\infty}\mathcal{F}_n(\tau,\tau',x,x')=h_{[\tau],[\tau']}(x,x')
\end{align}
uniformly on  $(\tau,\tau',x,x')\in[0,1]\times[0,1]\times M\times
M$.

Now we prove a preliminary result:

\begin{lemma}\label{le2-1} For each $n\in\mathbb{N}$ and each
$\tau\in[0,1]$,

\begin{itemize}
    \item [(1)] $h_{0,0}(x,z)\leq F_{0,n}(x,y)+h_{0,0}(y,z), \quad \forall
                 x,\ y,\ z\in M$;\\
    \item [(2)] $h_{0,[\tau]}(x,z)\leq h_{0,0}(x,y)+F_{0,n+\tau}(y,z), \quad
                 \forall x,\ y,\ z\in M$.
\end{itemize}
\end{lemma}

\begin{proof}
(1) Since $h_{0,0}(y,z)=\liminf_{k\in\mathbb{N} \atop
k\to+\infty}F_{0,k}(y,z)$, then there exist
$\{k_i\}_{i=1}^{+\infty}\subset\mathbb{N}$ such that
$k_i\to+\infty$ and $F_{0,k_i}(y,z)\to h_{0,0}(y,z)$ as
$i\to+\infty$. For each $n\in\mathbb{N}$ and each $k_i$, in view
of the definition of $F_{t,t'}$, we have

\[
F_{0,n+k_i}(x,z)\leq F_{0,n}(x,y)+F_{n,n+k_i}(y,z).
\]
Hence,

\[
h_{0,0}(x,z)\leq\liminf_{i\to+\infty}F_{0,n+k_i}(x,z)\leq
F_{0,n}(x,y)+h_{0,0}(y,z).
\]
(2) Since $h_{0,0}(x,y)=\liminf_{k\in\mathbb{N} \atop
k\to+\infty}F_{0,k}(x,y)$, then there exist
$\{k_i\}_{i=1}^{+\infty}\subset\mathbb{N}$ such that
$k_i\to+\infty$ and $F_{0,k_i}(x,y)\to h_{0,0}(x,y)$ as
$i\to+\infty$. For each $n\in\mathbb{N}$, each $\tau\in[0,1]$ and
each $k_i$, we have

\[
F_{0,\tau+n+k_i}(x,z)\leq F_{0,k_i}(x,y)+F_{k_i,\tau+n+k_i}(y,z).
\]
It follows that

\[
h_{0,[\tau]}(x,z)\leq\liminf_{i\to+\infty}F_{0,\tau+n+k_i}(x,z)\leq
h_{0,0}(x,y)+F_{0,n+\tau}(y,z).
\]
\end{proof}

Following Fathi \cite{Fat1}, as done by Contreras et al. in
\cite{Con}, we give the definition of the weak KAM solution as
follows.

\begin{definition}
A backward weak KAM solution of the Hamilton-Jacobi equation
(\ref{1-1}) is a function $w:M\times\mathbb{S}^1\to\mathbb{R}^1$
such that
\begin{itemize}
    \item [(1)] $w$ is dominated by $L$, i.e.,
               \[
               w(x_1,s_1)-w(x_2,s_2)\leq\Phi_{s_2,s_1}(x_2,x_1), \quad
               \forall (x_1,s_1),\ (x_2,s_2)\in M\times\mathbb{S}^1.
               \]
    \item [(2)] For every $(x,s)\in M\times\mathbb{S}^1$ there
               exists a curve $\gamma:(-\infty,\tilde{s}]\to M$ with
               $\gamma(\tilde{s})=x$ and $[\tilde{s}]=s$ such that
               \[
               w(x,s)-w(\gamma(t),[t])=\int_t^{\tilde{s}}L(\gamma(\sigma),\dot{\gamma}(\sigma),\sigma)d\sigma,\quad
               \forall t\in(-\infty,\tilde{s}].
               \]
\end{itemize}

Similarly, we say that $w:M\times\mathbb{S}^1\to\mathbb{R}^1$ is a
forward weak KAM solution of (\ref{1-1}) if $w$ is dominated by
$L$ and for every $(x,s)\in M\times\mathbb{S}^1$ there exists a
curve $\gamma:[\tilde{s},+\infty)\to M$ with $\gamma(\tilde{s})=x$
and $[\tilde{s}]=s$ such that
$w(\gamma(t),[t])-w(x,s)=\int^t_{\tilde{s}}L(\gamma(\sigma),\dot{\gamma}(\sigma),\sigma)d\sigma,
\forall t\in[\tilde{s},+\infty)$.
\end{definition}

We denote by $\mathcal{S}_-$ ($\mathcal{S}_+$) the set of backward
(forward) weak KAM solutions. The following well-known result
\cite{Con} will be used later.

\begin{lemma}\label{le2-2}
Given $(x_0,s_0)\in M\times\mathbb{S}^1$, define

\[
w^*(x,s):=h_{s_0,s}(x_0,x),\quad w_*(x,s):=-h_{s,s_0}(x,x_0)
\]
for $(x,s)\in M\times\mathbb{S}^1$. Then $w^*\in\mathcal{S}_-$,
$w_*\in \mathcal{S}_+$.
\end{lemma}

Define the projected Aubry set $\mathcal{A}_0$ as follows:

\[
\mathcal{A}_0:=\{(x,s)\in M\times\mathbb{S}^1\ |\
h_{s,s}(x,x)=0\}.
\]
Note that $\mathcal{A}_0=\Pi\tilde{\mathcal{A}_0}$, where
$\Pi:TM\times\mathbb{S}^1\to M\times\mathbb{S}^1$ denotes the
projection and $\tilde{\mathcal{A}_0}$ denotes the Aubry set in
$TM\times\mathbb{S}^1$. Define an equivalence relation on
$\mathcal{A}_0$ by saying that $(x_1,s_1)$ and $(x_2,s_2)$ are
equivalent if and only if

\begin{align*}
\Phi_{s_1,s_2}(x_1,x_2)+\Phi_{s_2,s_1}(x_2,x_1)=0.
\end{align*}

The equivalent classes of this relation are called static classes.
Let $\mathrm A$ be the set of static classes. For each static
class $\Gamma\in \mathrm A$ choose a point $(x,0)\in\Gamma$ and
let $\mathbb{A}_0$ be the set of such points. Since the Lagrangian
$L$ in Theorem \ref{th} satisfies (H5), then $\mathbb{A}_0$ consists
of only one point, denoted by $(p,0)\in\mathcal{A}_0$.

From a result of Contreras et al. \cite{Con}, for each backward
weak KAM solution $w$ of (\ref{1-1}), we have

\begin{align}\label{2-2}
w(x,s)=\min_{(q,0)\in\mathbb{A}_0}(w(q,0)+h_{0,s}(q,x))=w(p,0)+h_{0,s}(p,x)
\end{align}
for all $(x,s)\in M\times\mathbb{S}^1$.

Given $u\in C(M,\mathbb{R}^1)$, for each $n\in\mathbb{N}$, each
$\tau\in[0,1]$ and each $x\in M$,

\begin{align}\label{2-3}
\tilde{T}_n^\tau u(x)= \inf_{k\in\mathbb{N} \atop n\leq k\leq
2n}\inf_{\gamma}\Big\{u(\gamma(0))+\int_{0}^{\tau+k}
L(\gamma(s),\dot{\gamma}(s),s)ds\Big\},
\end{align}
where the second infimum is taken among the continuous and
piecewise $C^1$ paths $\gamma:[0,\tau+k]\rightarrow M$ with
$\gamma(\tau+k)=x$. In view of (\ref{2-3}), it is easy to see that
there exist $n\leq k_{x,\tau,n}\leq2n$,
$k_{x,\tau,n}\in\mathbb{N}$ and a minimizing extremal curve
$\gamma_{x,\tau,n}:[0,\tau+k_{x,\tau,n}]\to M$ such that
$\gamma_{x,\tau,n}(\tau+k_{x,\tau,n})=x$ and

\[
\tilde{T}_n^\tau
u(x)=u(\gamma_{x,\tau,n}(0))+\int_{0}^{\tau+k_{x,\tau,n}}
L(\gamma_{x,\tau,n}(s),\dot{\gamma}_{x,\tau,n}(s),s)ds.
\]
Let

\[
A(\gamma_{x,\tau,n}):=\int_{0}^{\tau+k_{x,\tau,n}}
L(\gamma_{x,\tau,n}(s),\dot{\gamma}_{x,\tau,n}(s),s)ds.
\]
Obviously, we have

\begin{align}\label{2-4}
A(\gamma_{x,\tau,n})=F_{0,\tau+k_{x,\tau,n}}(\gamma_{x,\tau,n}(0),x)=\mathcal{F}_n(0,\tau,\gamma_{x,\tau,n}(0),x).
\end{align}
The following result for minimizers $\gamma_{x,\tau,n}$ will be
used in the proof of Theorem \ref{th}.

\begin{lemma}\label{le2-3}
Under the assumptions of Theorem \ref{th}, let $W$ be a neighborhood of the Aubry set $\tilde{\mathcal{A}_0}$
in $TM\times\mathbb{S}^1$. Then there exists $T>0$ such that if
$n\geq T$, $n\in\mathbb{N}$, then

\[
\Big(\gamma_{x,\tau,n}(s),\dot{\gamma}_{x,\tau,n}(s),[s]\Big)\Big|_{[\frac{n}{3},\frac{2n}{3}]}\subset
W, \quad \forall (x,\tau)\in M\times [0,1].
\]
\end{lemma}

\begin{proof}
To prove the lemma, we argue by contradiction. For, otherwise,
there would be $\{n_i\}_{i=1}^{+\infty}\subset\mathbb{N}$ with
$n_i\to+\infty$ as $i\to+\infty$,
$\{(x_{n_i},\tau_{n_i})\}_{i=1}^{+\infty}\subset M\times [0,1]$,
$\{k_{n_i}\}_{i=1}^{+\infty}\subset\mathbb{N}$ with $n_i\leq
k_{n_i}\leq 2n_i$, a sequence
$\gamma_{n_i}:[0,\tau_{n_i}+k_{n_i}]\to M$, $i=1,2,\cdots$ of
minimizers satisfying $\gamma_{n_i}(\tau_{n_i}+k_{n_i})=x_{n_i}$
and $\tilde{T}_{n_i}^{\tau_{n_i}}
u(x_{n_i})=u(\gamma_{n_i}(0))+\int_{0}^{\tau_{n_i}+k_{n_i}}
L(\gamma_{n_i},\dot{\gamma}_{n_i},s)ds$, and
$\{t_{n_i}\}_{i=1}^{+\infty}$ with $\frac{n_i}{3}\leq
t_{n_i}\leq\frac{2n_i}{3}$ such that

\begin{align}\label{2-5}
\Big(\gamma_{n_i}(t_{n_i}),\dot{\gamma}_{n_i}(t_{n_i}),[t_{n_i}]\Big)\notin
W, \quad i=1,2,\cdots,
\end{align}
where we used $k_{n_i}$ and $\gamma_{n_i}$ to denote
$k_{x_{n_i},\tau_{n_i},n_i}$ and $\gamma_{x_{n_i},\tau_{n_i},n_i}$
respectively. For each positive integer $i$, we set
$y_{n_i}=\gamma_{n_i}(0)$. Passing as necessary to a subsequence,
we may suppose that $x_{n_i}\to x_0$, $y_{n_i}\to y_0$ and
$\tau_{n_i}\to \tau_0$ as $i\to+\infty$, where $x_0$, $y_0\in M$
and $\tau_0\in[0,1]$.

Since

\begin{align*}
\big|\mathcal{F}_{n_i}(0,\tau_{n_i},y_{n_i},x_{n_i})-h_{0,[\tau_0]}(y_0,x_0)\big|
& \leq
\big|\mathcal{F}_{n_i}(0,\tau_{n_i},y_{n_i},x_{n_i})-h_{0,[\tau_{n_i}]}(y_{n_i},x_{n_i})\big|\\
& +
\big|h_{0,[\tau_{n_i}]}(y_{n_i},x_{n_i})-h_{0,[\tau_0]}(y_{n_i},x_{n_i})\big|\\
& +
\big|h_{0,[\tau_0]}(y_{n_i},x_{n_i})-h_{0,[\tau_0]}(y_0,x_0)\big|,
\end{align*}
then from (\ref{2-1}) and the Lipschitz property of $h$, we have

\begin{align}\label{2-6}
\lim_{i\to+\infty}\mathcal{F}_{n_i}(0,\tau_{n_i},y_{n_i},x_{n_i})=h_{0,[\tau_0]}(y_0,x_0).
\end{align}
In view of (\ref{2-4}) and (\ref{2-6}), we have

\begin{align}\label{2-7}
\lim_{i\to+\infty}A(\gamma_{n_i})=h_{0,[\tau_0]}(y_0,x_0).
\end{align}

For each $i$, we set
$(\tilde{x}_{n_i},\dot{\tilde{x}}_{n_i},\sigma_{n_i})=(\gamma_{n_i}(t_{n_i}),\dot{\gamma}_{n_i}(t_{n_i}),[t_{n_i}])$.
By (\ref{2-5}),
$(\tilde{x}_{n_i},\dot{\tilde{x}}_{n_i},\sigma_{n_i})\notin W$,
$\forall i$. Since $\gamma_{n_i}$ are minimizing extremal curves,
using the a priori compactness Lemma 3.4 in \cite{W-Y}, we
conclude that
$\{(\tilde{x}_{n_i},\dot{\tilde{x}}_{n_i},\sigma_{n_i})\}_{i=1}^{+\infty}$
are contained in a compact subset of $TM\times\mathbb{S}^1$. So we
may assume upon passing if necessary to a subsequence that
$(\tilde{x}_{n_i},\dot{\tilde{x}}_{n_i},\sigma_{n_i})\to
(\tilde{x},\dot{\tilde{x}},\sigma)\in TM\times\mathbb{S}^1$ as
$i\to+\infty$. Since
$(\tilde{x}_{n_i},\dot{\tilde{x}}_{n_i},\sigma_{n_i})\notin W$,
$\forall i$, then
$(\tilde{x},\dot{\tilde{x}},\sigma)\notin\tilde{\mathcal{A}_0}$.

Let
$(\gamma(s),\dot{\gamma}(s),[s])=\phi^L_{s-\sigma}(\tilde{x},\dot{\tilde{x}},\sigma)$,
$s\in\mathbb{R}^1$. We assert that the orbit
$(\gamma(s),\dot{\gamma}(s),[s])$ is global semi-static, i.e.,
$\gamma$ is a global semi-static curve. If this assertion is true,
then $(\tilde{x},\dot{\tilde{x}},\sigma)\in
\tilde{\mathcal{N}}_0$. By our assumption that
$\tilde{\mathcal{A}}_0$ consists of one hyperbolic 1-periodic
orbit, it is easy to see that
$\tilde{\mathcal{M}}_0=\tilde{\mathcal{A}}_0=\tilde{\mathcal{N}}_0$.
Thus, we deduce that
$(\tilde{x},\dot{\tilde{x}},\sigma)\in\tilde{\mathcal{A}_0}$,
which is impossible since
$(\tilde{x},\dot{\tilde{x}},\sigma)\notin\tilde{\mathcal{A}_0}$.
This contradiction proves the lemma.

Based on the above arguments, it is sufficient to show that
$\gamma$ is a global semi-static curve. We prove it by
contradiction. Otherwise, there would be $j_1$, $j_2\in\mathbb{N}$
such that

\[
A(\gamma\big|_{[\sigma-j_1,\sigma+j_2]})>\Phi_{\sigma,\sigma}(\gamma(\sigma-j_1),\gamma(\sigma+j_2)).
\]
It implies that there exist $j'_1$, $j'_2\in\mathbb{N}$ with
$\sigma-j'_1+1\leq\sigma+j'_2$ and a minimizing curve
$\tilde{\gamma}:[\sigma-j'_1,\sigma+j'_2]\to M$ satisfying
$\tilde{\gamma}(\sigma-j'_1)=\gamma(\sigma-j_1)$ and
$\tilde{\gamma}(\sigma+j'_2)=\gamma(\sigma+j_2)$ such that

\[
A(\gamma\big|_{[\sigma-j_1,\sigma+j_2]})>A(\tilde{\gamma}\big|_{[\sigma-j'_1,\sigma+j'_2]}).
\]
Thus, there exists $\Delta>0$ such that

\begin{align}\label{2-8}
A(\tilde{\gamma}\big|_{[\sigma-j'_1,\sigma+j'_2]})\leq
A(\gamma\big|_{[\sigma-j_1,\sigma+j_2]})-\Delta.
\end{align}

Since $(\tilde{x}_{n_i},\dot{\tilde{x}}_{n_i},\sigma_{n_i})\to
(\tilde{x},\dot{\tilde{x}},\sigma)\in TM\times\mathbb{S}^1$ as
$i\to+\infty$, then, for every $\varepsilon>0$, by the differentiability
of the solutions of the Euler-Lagrange equation with respect to
initial values, we have

\begin{align}\label{2-9}
d\big((\gamma(s),\dot{\gamma}(s),[s]),(\gamma_{n_i}(s+t_{n_i}-\sigma),\dot{\gamma}_{n_i}(s+t_{n_i}-\sigma),[s+t_{n_i}-\sigma])\big)<\varepsilon,
\end{align}
for all $s\in[\sigma-j_1,\sigma+j_2]$ and $i$ large enough. Using the
periodicity of $L$, we have

\begin{align}\label{2-10}
A(\gamma_{n_i}\big|_{[t_{n_i}-j_1,t_{n_i}+j_2]})=\int_{\sigma-j_1}^{\sigma+j_2}L(\gamma_{n_i}(s+t_{n_i}-\sigma),\dot{\gamma}_{n_i}(s+t_{n_i}-\sigma),[s+t_{n_i}-\sigma])ds,
\end{align}
and

\begin{align}\label{2-11}
A(\gamma\big|_{[\sigma-j_1,\sigma+j_2]})=\int_{\sigma-j_1}^{\sigma+j_2}L(\gamma(s),\dot{\gamma}(s),[s])ds.
\end{align}
Combining (\ref{2-9}), (\ref{2-10}) and (\ref{2-11}), we have

\begin{align}\label{2-12}
\Big|A(\gamma_{n_i}\big|_{[t_{n_i}-j_1,t_{n_i}+j_2]})-A(\gamma\big|_{[\sigma-j_1,\sigma+j_2]})\big|\leq
C\varepsilon
\end{align}
for some constant $C>0$ independent of $\varepsilon$ and
sufficiently large $i$. Since $\varepsilon$ may be taken arbitrary
small, from (\ref{2-8}) and (\ref{2-12}) we obtain

\begin{align}\label{2-13}
A(\gamma_{n_i}\big|_{[t_{n_i}-j_1,t_{n_i}+j_2]})\geq
A(\gamma\big|_{[\sigma-j_1,\sigma+j_2]})-C\varepsilon\geq
A(\tilde{\gamma}\big|_{[\sigma-j'_1,\sigma+j'_2]})+\frac{3\Delta}{4},
\end{align}
provided $i$ is large enough.

We set $\bar{x}=\tilde{\gamma}(\sigma-j'_1)=\gamma(\sigma-j_1)$
and
$\underline{x}=\tilde{\gamma}(\sigma+j'_2)=\gamma(\sigma+j_2)$.
For $i$ large enough, consider the following curves on $M$. Let
$\alpha^1_i:[0,t_{n_i}-j_1]\to M$ with $\alpha^1_i(0)=y_{n_i}$ and
$\alpha^1_i(t_{n_i}-j_1)=\bar{x}$ be a Tonelli minimizer such that

\[
A(\alpha^1_i)=F_{0,t_{n_i}-j_1}(y_{n_i},\bar{x}).
\]
Let
$\alpha^2_i:[t_{n_i}-j_1+j'_1+j'_2,\tau_{n_i}+k_{n_i}-j_1-j_2+j'_1+j'_2]\to
M$ with $\alpha^2_i(t_{n_i}-j_1+j'_1+j'_2)=\underline{x}$ and
$\alpha^2_i(\tau_{n_i}+k_{n_i}-j_1-j_2+j'_1+j'_2)=x_{n_i}$ be a
Tonelli minimizer such that

\[
A(\alpha^2_i)=F_{t_{n_i}-j_1+j'_1+j'_2,\tau_{n_i}+k_{n_i}-j_1-j_2+j'_1+j'_2}(\underline{x},x_{n_i}).
\]
Let

\[
\tilde{\gamma}_{n_i}(s)= \left\{\begin{array}{ll}
          \alpha^1_i(s),\ & s\in[0,t_{n_i}-j_1],\\[2mm]
          \tilde{\gamma}(s-t_{n_i}+j_1+\sigma-j'_1),\ & s\in[t_{n_i}-j_1,t_{n_i}-j_1+j'_1+j'_2],\\[2mm]
          \alpha^2_i(s),\ & s\in[t_{n_i}-j_1+j'_1+j'_2,\tau_{n_i}+k_{n_i}-j_1-j_2+j'_1+j'_2].
\end{array}\right.
\]
It is clearly that
$\tilde{\gamma}_{n_i}:[0,\tau_{n_i}+k_{n_i}-j_1-j_2+j'_1+j'_2]\to
M$ is a continuous and piecewise $C^1$ curve connecting $y_{n_i}$
and $x_{n_i}$.

We set $\bar{x}_{n_i}=\gamma_{n_i}(t_{n_i}-j_1)$ and
$\underline{x}_{n_i}=\gamma_{n_i}(t_{n_i}+j_2)$. For $i$ large
enough, compare $A(\tilde{\gamma}_{n_i})$ with $A(\gamma_{n_i})$
as follows. In view of (\ref{2-9}), we have

\begin{align}\label{2-14}
\Big|A(\tilde{\gamma}_{n_i}\big|_{[0,t_{n_i}-j_1]})-A(\gamma_{n_i}\big|_{[0,t_{n_i}-j_1]})\Big|
=\Big|F_{0,t_{n_i}-j_1}(y_{n_i},\bar{x})-F_{0,t_{n_i}-j_1}(y_{n_i},\bar{x}_{n_i})\Big|\leq
D_{\mathrm{Lip}}\varepsilon,
\end{align}
where $D_{\mathrm{Lip}}>0$ is a Lipschitz constant of $F_{t,t'}$
which is independent of $t$, $t'$ with $t+1\leq t'$ \cite{Ber}.
Note that

\begin{align}\label{2-15}
\begin{split}
&A(\tilde{\gamma}_{n_i}\big|_{[t_{n_i}-j_1,t_{n_i}-j_1+j'_1+j'_2]})-A(\gamma_{n_i}\big|_{[t_{n_i}-j_1,t_{n_i}+j_2]})\\
&=\int_{\sigma-j'_1}^{\sigma+j'_2}L(\tilde{\gamma}(s),\dot{\tilde{\gamma}}(s),s+\sigma_{n_i}-\sigma)ds-A(\gamma_{n_i}\big|_{[t_{n_i}-j_1,t_{n_i}+j_2]}).
\end{split}
\end{align}
Since $\sigma_{n_i}\to\sigma$ as $i\to+\infty$, then

\begin{align}\label{2-16}
\Big|A(\tilde{\gamma}\big|_{[\sigma-j'_1,\sigma+j'_2]})-\int_{\sigma-j'_1}^{\sigma+j'_2}L(\tilde{\gamma}(s),\dot{\tilde{\gamma}}(s),s+\sigma_{n_i}-\sigma)ds\Big|
\leq\frac{\Delta}{4}
\end{align}
for $i$ large enough. Hence, from (\ref{2-13}), (\ref{2-15}) and
(\ref{2-16}) we have

\begin{align}\label{2-17}
A(\tilde{\gamma}_{n_i}\big|_{[t_{n_i}-j_1,t_{n_i}-j_1+j'_1+j'_2]})-A(\gamma_{n_i}\big|_{[t_{n_i}-j_1,t_{n_i}+j_2]})\leq-\frac{\Delta}{2}.
\end{align}
From the Lipschitz property of $F_{t,t'}$ and (\ref{2-9}), we find

\begin{align}\label{2-18}
\begin{split}
&
\Big|A(\tilde{\gamma}_{n_i}\big|_{[t_{n_i}-j_1+j'_1+j'_2,\tau_{n_i}+k_{n_i}-j_1-j_2+j'_1+j'_2]})-A(\gamma_{n_i}\big|_{[t_{n_i}+j_2,\tau_{n_i}+k_{n_i}]})\Big|\\
&
=\Big|F_{t_{n_i}-j_1+j'_1+j'_2,\tau_{n_i}+k_{n_i}-j_1-j_2+j'_1+j'_2}(\underline{x},x_{n_i})-F_{t_{n_i}+j_2,\tau_{n_i}+k_{n_i}}(\underline{x}_{n_i},x_{n_i})\Big|\\
& \leq D_{\mathrm{Lip}}\varepsilon.
\end{split}
\end{align}

Since $\varepsilon$ may be taken arbitrary small, from
(\ref{2-14}), (\ref{2-17}) and (\ref{2-18}), we have

\begin{align}\label{2-19}
A(\tilde{\gamma}_{n_i})\leq A(\gamma_{n_i})-\frac{\Delta}{4}
\end{align}
for $i$ large enough.

For each sufficiently large $i$, we choose $m_i\in\mathbb{N}$ such
that

\begin{align}\label{2-20}
m_i\leq k_{n_i}-j_1-j_2+j'_1+j'_2\leq 2m_i.
\end{align}
Since $n_i\leq k_{n_i}\leq 2n_i$, $n_i\to+\infty$ as
$i\to+\infty$, then $m_i\to+\infty$ as $i\to+\infty$. By
(\ref{2-20}), for each $i$ large enough, we have

\begin{align}\label{2-21}
A(\tilde{\gamma}_{n_i})\geq
F_{0,\tau_{n_i}+k_{n_i}-j_1-j_2+j'_1+j'_2}(y_{n_i},x_{n_i})\geq
\mathcal{F}_{m_i}(0,\tau_{n_i},y_{n_i},x_{n_i}).
\end{align}

Since

\begin{align*}
\big|\mathcal{F}_{m_i}(0,\tau_{n_i},y_{n_i},x_{n_i})-h_{0,[\tau_0]}(y_0,x_0)\big|
& \leq
\big|\mathcal{F}_{m_i}(0,\tau_{n_i},y_{n_i},x_{n_i})-h_{0,[\tau_{n_i}]}(y_{n_i},x_{n_i})\big|\\
& +
\big|h_{0,[\tau_{n_i}]}(y_{n_i},x_{n_i})-h_{0,[\tau_0]}(y_{n_i},x_{n_i})\big|\\
& +
\big|h_{0,[\tau_0]}(y_{n_i},x_{n_i})-h_{0,[\tau_0]}(y_0,x_0)\big|,
\end{align*}
then from (\ref{2-1}) and the Lipschitz property of $h$, we have

\begin{align}\label{2-22}
\lim_{i\to+\infty}\mathcal{F}_{m_i}(0,\tau_{n_i},y_{n_i},x_{n_i})=h_{0,[\tau_0]}(y_0,x_0).
\end{align}

Combining (\ref{2-7}), (\ref{2-19}), (\ref{2-21}) and
(\ref{2-22}), we have

\begin{align*}
h_{0,[\tau_0]}(y_0,x_0)-\frac{\Delta}{4}
&=\lim_{i\to+\infty}A(\gamma_{n_i})-\frac{\Delta}{4}\\
&\geq\liminf_{i\to+\infty}A(\tilde{\gamma}_{n_i})\\
&\geq\lim_{i\to+\infty}\mathcal{F}_{m_i}(0,\tau_{n_i},y_{n_i},x_{n_i})\\
&=h_{0,[\tau_0]}(y_0,x_0),
\end{align*}
a contradiction. This contradiction shows that $\gamma$ is global
semi-static.
\end{proof}

\begin{remark}
The above result is independent of $u\in C(M,\mathbb{R}^1)$.
Moreover, from the proof, it is easy to see that the conclusion of
Lemma \ref{le2-3} holds with $[\frac{n}{3},\frac{2n}{3}]$ replaced
by $[an,bn]$ for arbitrary $0<a<b<1$.
\end{remark}

\section{Proof of the main result}
In this section we prove Theorem \ref{th}. Let $(p,v_p,0)$ be the
unique point in $\tilde{\mathcal{A}}_0$ with
$\Pi(p,v_p,0)=(p,0)\in \mathcal{A}_0$, where
$\Pi:TM\times\mathbb{S}^1\to M\times\mathbb{S}^1$ denotes the
projection. By (H5) the Aubry set $\tilde{\mathcal{A}}_0$ consists
of one hyperbolic 1-periodic orbit, denoted by
$\Gamma:\phi^L_s(p,v_p,0)=(\gamma_p(s),\dot{\gamma}_p(s),[s])$,
$s\in\mathbb{R}^1$.

\vskip0.2cm

\noindent\emph{Proof of Theorem \ref{th}.} Our purpose is to show
that there exists $\rho>0$ such that for each $u\in
C(M,\mathbb{R}^1)$ there is $K>0$ such that the following two
inequalities hold.

\[
\bar{u}(x,[\tau])-U^u_n(x,\tau)\leq Ke^{-\rho n}, \quad \forall
n\in\mathbb{N},\ \forall (x,\tau)\in M\times[0,1]; \eqno
(\mathrm{I1})
\]

\[
U^u_n(x,\tau)-\bar{u}(x,[\tau])\leq Ke^{-\rho n}, \quad \forall
n\in\mathbb{N},\ \forall (x,\tau)\in M\times[0,1]. \eqno
(\mathrm{I2})
\]
Since the proof is rather long, it is convenient to divide it into
two steps.

\vskip0.1cm

Step 1. We first prove the inequality (I1). In view of Lemma
\ref{le2-2} and (\ref{2-2}), for any given $y\in M$,
$h_{0,\cdot}(y,\cdot)$ is a backward weak KAM solution of
(\ref{1-1}) and

\begin{align}\label{3-1}
h_{0,[\tau]}(y,x)=h_{0,0}(y,p)+h_{0,[\tau]}(p,x)
\end{align}
for all $(x,\tau)\in M\times[0,1]$. Given $u\in C(M,\mathbb{R}^1)$
and $(x,\tau)\in M\times[0,1]$, from Theorem \ref{oth} and
(\ref{3-1}) we have

\begin{align}\label{3-2}
\bar{u}(x,[\tau])=\inf_{y\in M}(u(y)+h_{0,[\tau]}(y,x))=\inf_{y\in
M}(u(y)+h_{0,0}(y,p)+h_{0,[\tau]}(p,x)).
\end{align}

By the arguments in Section 2, for each $n\in\mathbb{N}$ there
exist $n\leq k_{x,\tau,n}\leq2n$, $k_{x,\tau,n}\in\mathbb{N}$ and
a minimizing extremal curve
$\gamma_{x,\tau,n}:[0,\tau+k_{x,\tau,n}]\to M$ such that
$\gamma_{x,\tau,n}(\tau+k_{x,\tau,n})=x$ and

\begin{align}\label{3-3}
\tilde{T}_n^\tau
u(x)=u(\gamma_{x,\tau,n}(0))+\int_{0}^{\tau+k_{x,\tau,n}}
L(\gamma_{x,\tau,n}(s),\dot{\gamma}_{x,\tau,n}(s),s)ds.
\end{align}
In what follows we use $k_n$ and $\gamma_n$ to denote
$k_{x,\tau,n}$ and $\gamma_{x,\tau,n}$ respectively.

From (\ref{3-2}) and Lemma \ref{le2-1}, we have

\begin{align}\label{3-4}
\begin{split}
\bar{u}(x,[\tau])
& \leq
u(\gamma_n(0))+h_{0,0}(\gamma_n(0),p)+h_{0,[\tau]}(p,x)\\
& \leq
u(\gamma_n(0))+F_{0,n_1}(\gamma_n(0),\gamma_n(s))+h_{0,0}(\gamma_n(s),p)\\
&+h_{0,0}(p,\gamma_n(s))+F_{0,n_2+\tau}(\gamma_n(s),x)
\end{split}
\end{align}
for all $s\in[0,\tau+k_n]$ and all $n_1$, $n_2\in\mathbb{N}$. For
$n\in\mathbb{N}$ large enough, let
$j_n=\{\frac{2n}{3}\}-\{\frac{n}{3}\}-1$. Taking
$n_1=\{\frac{j_n}{2}\}+\{\frac{n}{3}\}+1$, $s=n_1$ and
$n_2=k_n-n_1$, by (\ref{3-4}) we obtain

\begin{align}\label{3-5}
\begin{split}
\bar{u}(x,[\tau])\leq
u(\gamma_n(0))+\int_0^{\tau+k_n}L(\gamma_n,\dot{\gamma}_n,s)ds
+2C_{\mathrm{Lip}}d\Big(\gamma_n\big(\{\frac{j_n}{2}\}+\{\frac{n}{3}\}+1\big),p\Big),
\end{split}
\end{align}
where $C_{\mathrm{Lip}}$ is a Lipschitz constant of $h$. From
(\ref{3-3}) and (\ref{3-5}) we have

\begin{align}\label{3-6}
\bar{u}(x,[\tau])-U^u_n(x,\tau)\leq
2C_{\mathrm{Lip}}d\Big(\gamma_n\big(\{\frac{j_n}{2}\}+\{\frac{n}{3}\}+1\big),p\Big).
\end{align}

We now estimate the term in the right-hand side of (\ref{3-6}).
Consider the Poincar\'e map for the time-periodic Lagrangian
system $L$:

\[
\varphi_{1,0}: TM\to TM, \quad
(x_0,v_0)\mapsto\varphi_{1,0}(x_0,v_0),
\]
where $\varphi_{t,0}(x_0,v_0)=(x(t),\dot{x}(t))$ and $x(t)$
denotes the solution to the Euler-Lagrange equation with initial
conditions $x(0)=x_0$, $\dot{x}(0)=v_0$. Obviously, $\phi^L_t(x_0,v_0,0)
=(\varphi_{t,0}(x_0,v_0),[t])$.  It is easy to see that
$(p,v_p)$ is a hyperbolic fixed point of $\varphi_{1,0}$.
According to the Hartman-Grobman Theorem the Poincar\'e map
$\varphi_{1,0}$ is locally conjugate to its linear part at the
hyperbolic fixed point $(p,v_p)$. More precisely, there exist a
neighborhood $V(p,v_p)$ of $(p,v_p)$ in $TM$ as well as a
neighborhood $U(0)$ of 0 in $T_{(p,v_p)}(TM)$ and a homeomorphism
$f:V(p,v_p)\to U(0)$, such that

\begin{align}\label{3-7}
D\varphi_{1,0}(p,v_p)\circ f=f\circ\varphi_{1,0}.
\end{align}
Furthermore, there exists $0<\alpha<1$ such that $f$ and $f^{-1}$
are $\alpha$-H$\mathrm{\ddot{o}}$lder continuous \cite{Bar}
\cite{Bel}. Denote for brevity $P=(p,v_p)$. As the problem here is
a local one we can, using a local chart, suppose that
$\varphi_{1,0}$ is a map from $\mathbb{R}^{2m}$ to itself with $P$
as a hyperbolic fixed point.

Let $B(P)$ be a sufficiently small neighborhood of $P$ in
$\mathbb{R}^{2m}$ such that $B(P)\subset V(P)=V(p,v_p)$. We choose
a tubular neighborhood $W_\Gamma$ of $\Gamma$ such that for each
$(q,v,[s])\in \Gamma$, $d((q,v,[s]),\partial W_\Gamma)=\delta$,
where $\partial W_\Gamma$ denotes the boundary of $W_\Gamma$ and
$\delta$ is a positive constant small enough such that for each
$(q,v,0)\in W_\Gamma$, we have $(q,v)\in B(P)$. For the tubular
neighborhood $W_\Gamma$, applying Lemma \ref{le2-3}, there exists
$T>0$ such that for $n\in\mathbb{N}$ with $n\geq T$, we have

\[
\Big(\gamma_n(s),\dot{\gamma}_n(s),[s]\Big)\Big|_{[\frac{n}{3},\frac{2n}{3}]}\subset
W_\Gamma.
\]
It follows that

\[
\Big(\gamma_n\big(\{\frac{n}{3}\}+1\big),\dot{\gamma}_n\big(\{\frac{n}{3}\}+1\big),0\Big),\cdots,\Big(\gamma_n\big(\{\frac{2n}{3}\}\big),\dot{\gamma}_n\big(\{\frac{2n}{3}\}\big),0\Big)\in
W_\Gamma.
\]
Thus, we have

\[
\Big(\gamma_n\big(\{\frac{n}{3}\}+1\big),\dot{\gamma}_n\big(\{\frac{n}{3}\}+1\big)\Big),\cdots,\Big(\gamma_n\big(\{\frac{2n}{3}\}\big),\dot{\gamma}_n\big(\{\frac{2n}{3}\}\big)\Big)\in
B(P),
\]
i.e.,

\begin{align}\label{3-8}
\varphi_{1,0}^{\{\frac{n}{3}\}+1}(P^n_1),\cdots,\varphi_{1,0}^{\{\frac{2n}{3}\}}(P^n_1)\in
B(P),
\end{align}
where $P^n_1=(\gamma_n(0),\dot{\gamma}_n(0))$. By (\ref{3-7}) and
(\ref{3-8}) we have

\[
Af(P^n_2)=f\circ\varphi_{1,0}^{\{\frac{n}{3}\}+2}(P^n_1),\cdots,A^{j_n}f(P^n_2)=f\circ\varphi_{1,0}^{\{\frac{2n}{3}\}}(P^n_1),
\]
where $A=D\varphi_{1,0}(P)$ and
$P^n_2=\varphi_{1,0}^{\{\frac{n}{3}\}+1}(P^n_1)$. In view of
(\ref{3-8}), we obtain

\[
A^if(P^n_2)\in U(0), \quad i=0,1,\cdots,j_n.
\]
Hence, there exists $\bar{\Delta}>0$ such that

\begin{align}\label{3-9}
\|A^if(P^n_2)\|\leq\bar{\Delta}, \quad i=0,1,\cdots,j_n.
\end{align}

As $A:\mathbb{R}^{2m}\to\mathbb{R}^{2m}$ is hyperbolic, there
exists an invariant splitting $\mathbb{R}^{2m}=E^s\oplus E^u$. For
each $z\in\mathbb{R}^{2m}$, we have $z=z_s+z_u$, $z_s\in E^s$,
$z_u\in E^u$ and $Az=A_sz_s+A_uz_u$, where $A_s=A|_{E^s}$ and
$A_u=A|_{E^u}$. Let $f(P^n_2)=y^n_s+y^n_u$, $y^n_s\in E^s$,
$y^n_u\in E^u$ and $A^{j_n}f(P^n_2)=z^n_s+z^n_u$, $z^n_s\in E^s$,
$z^n_u\in E^u$. Let $\lambda_1,\cdots,\lambda_m$ be the
eigenvalues of $A_s$. Then $|\lambda_i|<1$ for $i=1,\cdots,m$.
Since $A$ is similar to a symplectic matrix, then
$\frac{1}{\lambda_1},\cdots,\frac{1}{\lambda_m}$ are the
eigenvalues of $A_u$. Set $\lambda_\mathrm{max}=\max_{1\leq i\leq
m}|\lambda_i|$. It is a standard result that for arbitrary
$\varepsilon>0$, we have

\begin{align}\label{3-10}
\|A^i_sz_s\|\leq(\lambda_\mathrm{max}+\varepsilon)^i\|z_s\|, \quad
\forall z_s\in E^s,
\end{align}
for $i\in\mathbb{N}$ large enough. We choose $\varepsilon_0>0$
small enough such that $\lambda_\mathrm{max}+\varepsilon_0<1$.
Then from (\ref{3-10}) we have

\begin{align}\label{3-11}
\|A^{\{\frac{j_n}{2}\}}_sy^n_s\|\leq(\lambda_\mathrm{max}+\varepsilon_0)^{\{\frac{j_n}{2}\}}\|y^n_s\|\leq(\lambda_\mathrm{max}+\varepsilon_0)^{\{\frac{j_n}{2}\}}\bar{\Delta}
\end{align}
for $n$ large enough. Similarly, we have

\begin{align}\label{3-12}
\|A^{\{\frac{j_n}{2}\}}_uy^n_u\|=\|A^{-(j_n-\{\frac{j_n}{2}\})}_uz^n_u\|\leq(\lambda_\mathrm{max}+\varepsilon_0)^{j_n-\{\frac{j_n}{2}\}}\|z^n_u\|\leq(\lambda_\mathrm{max}+\varepsilon_0)^{\{\frac{j_n}{2}\}}\bar{\Delta}
\end{align}
for $n$ large enough. By (\ref{3-11}) and (\ref{3-12}), we obtain

\begin{align}\label{3-13}
\|A^{\{\frac{j_n}{2}\}}f(P^n_2)\|\leq\|A^{\{\frac{j_n}{2}\}}_sy^n_s\|+\|A^{\{\frac{j_n}{2}\}}_uy^n_u\|\leq
2\bar{\Delta}
(\lambda_\mathrm{max}+\varepsilon_0)^{\{\frac{j_n}{2}\}}
\end{align}
for $n$ large enough. Since
$j_n=\{\frac{2n}{3}\}-\{\frac{n}{3}\}-1$, then from (\ref{3-13})
we have

\begin{align}\label{3-14}
\|A^{\{\frac{j_n}{2}\}}f(P^n_2)\|\leq 2\bar{\Delta}
(\lambda_\mathrm{max}+\varepsilon_0)^{\frac{n}{12}}
\end{align}
for $n$ large enough. Note that
$A^{\{\frac{j_n}{2}\}}f(P^n_2)=f\circ\varphi_{1,0}^{\{\frac{j_n}{2}\}+\{\frac{n}{3}\}+1}(P^n_1)$
and $f(P)=0$. Since $f^{-1}$ is $\alpha$-H$\mathrm{\ddot{o}}$lder
continuous, from (\ref{3-14}) we have

\begin{align}\label{3-15}
\begin{split}
\|\varphi_{1,0}^{\{\frac{j_n}{2}\}+\{\frac{n}{3}\}+1}(P^n_1)-P\|
&=\|f^{-1}\circ A^{\{\frac{j_n}{2}\}}f(P^n_2)-f^{-1}(0)\|\\
&\leq C_1\|A^{\{\frac{j_n}{2}\}}f(P^n_2)-0\|^{\alpha}\\
& \leq
C_12^{\alpha}\bar{\Delta}^{\alpha}(\lambda_\mathrm{max}+\varepsilon_0)^{\frac{\alpha
n}{12}}
\end{split}
\end{align}
for $n$ large enough, where $C_1>0$ is a constant. Therefore,
there exists a constant $C_2>0$ independent of $u\in C(M,\mathbb{R}^1)$ and
$(x,\tau)\in M\times[0,1]$ such that

\begin{align}\label{3-16}
d\Big(\gamma_n\big(\{\frac{j_n}{2}\}+\{\frac{n}{3}\}+1\big),p\Big)\leq
C_2(\lambda_\mathrm{max}+\varepsilon_0)^{\frac{\alpha n}{12}}
\end{align}
for $n$ large enough. Note that the above estimate is independent
of $(x,\tau)$. By (\ref{3-6}) and (\ref{3-16}), for sufficiently
large $n$, we have

\[
\bar{u}(x,[\tau])-U^u_n(x,\tau)\leq
2C_{\mathrm{Lip}}C_2(\lambda_\mathrm{max}+\varepsilon_0)^{\frac{\alpha
n}{12}}, \quad \forall (x,\tau)\in M\times[0,1].
\]
Hence, there exists a constant $C_3>0$ such that

\[
\bar{u}(x,[\tau])-U^u_n(x,\tau)\leq
C_3(\lambda_\mathrm{max}+\varepsilon_0)^{\frac{\alpha n}{12}},
\quad \forall n\in\mathbb{N}, \ \forall (x,\tau)\in M\times[0,1],
\]
where the constant $C_3$ depends on $u$. Since
$0<\lambda_\mathrm{max}+\varepsilon_0<1$, there exists $\rho_1>0$
such that $(\lambda_\mathrm{max}+\varepsilon_0)^{\frac{\alpha
}{12}}=e^{-\rho_1}$. Thus, we have

\begin{align}\label{3-17}
\bar{u}(x,[\tau])-U^u_n(x,\tau)\leq C_3e^{-\rho_1n}, \quad \forall
n\in\mathbb{N}, \ \forall (x,\tau)\in M\times[0,1].
\end{align}

\vskip0.1cm

Step 2. We now prove the inequality (I2). Given $u\in
C(M,\mathbb{R}^1)$ and $(x,\tau)\in M\times[0,1]$, by (\ref{3-2})
we have

\[
\bar{u}(x,[\tau])=\inf_{z\in
M}(u(z)+h_{0,0}(z,p)+h_{0,[\tau]}(p,x)).
\]
Thus, there exists $y\in M$ such that

\begin{align}\label{3-18}
\bar{u}(x,[\tau])=u(y)+h_{0,0}(y,p)+h_{0,[\tau]}(p,x).
\end{align}

Since $h_{0,\cdot}(p,\cdot)$ is a backward weak KAM solution of
(\ref{1-1}), then there is a curve
$\beta_{x,[\tau]}:(-\infty,\tilde{\tau}]\to M$ with
$\beta_{x,[\tau]}(\tilde{\tau})=x$ and $[\tilde{\tau}]=[\tau]$
such that

\[
h_{0,[\tau]}(p,x)-h_{0,[t]}(p,\beta_{x,[\tau]}(t))=\int_t^{\tilde{\tau}}L(\beta_{x,[\tau]},\dot{\beta}_{x,[\tau]},s)ds,\quad
\forall t\in (-\infty,\tilde{\tau}].
\]
It is clear that $\beta_{x,[\tau]}$ is a minimizing curve and the
$\alpha$-limit set for
$(\beta_{x,[\tau]}(s),\dot{\beta}_{x,[\tau]}(s),[s])$ is $\Gamma$.
Similarly, since $-h_{\cdot,0}(\cdot,p)$ is a forward weak KAM
solution of (\ref{1-1}), then there exists a curve
$\omega_{y,0}:[\tilde{o},+\infty)\to M$ with
$\omega_{y,0}(\tilde{o})=y$ and $[\tilde{o}]=0$ such that

\[
h_{0,0}(y,p)-h_{[t],0}(\omega_{y,0}(t),p)=\int_{\tilde{o}}^tL(\omega_{y,0},\dot{\omega}_{y,0},s)ds,\quad
\forall t\in [\tilde{o},+\infty).
\]
Moreover, $\omega_{y,0}$ is a minimizing curve and the $\omega$-limit set
for $(\omega_{y,0}(s),\dot{\omega}_{y,0}(s),[s])$ is $\Gamma$.

Since $\Gamma$ is a hyperbolic 1-periodic orbit, then for the
tubular neighborhood $W_{\Gamma}$ there exist constants $T_1>0$ and $C_4>0$,
such that

\begin{align}\label{3-19}
d((\omega_{y,0}(s+\tilde{o}),\dot{\omega}_{y,0}(s+\tilde{o}),[s+\tilde{o}]),(\gamma_p(s),\dot{\gamma}_p(s),[s]))\leq
C_4e^{-\mu s}, \ s>T_1,
\end{align}
and

\begin{align}\label{3-20}
d((\beta_{x,[\tau]}(s+\tilde{\tau}),\dot{\beta}_{x,[\tau]}(s+\tilde{\tau}),[s+\tilde{\tau}]),(\gamma_p(s+[\tau]),\dot{\gamma}_p(s+[\tau]),[s+[\tau]]))\leq
C_4e^{\mu s}, \ s<-T_1,
\end{align}
where $T_1$ and $C_4$ depend only on $W_{\Gamma}$, and $\mu$
denotes the smallest positive Lyapunov exponent of $\Gamma$.

For $n\in\mathbb{N}$ large enough such that
$\frac{2n}{3}>\max\{T_1,2\}$, by (\ref{3-19}) we have

\[
d\Big(\big(\omega_{y,0}\big(\frac{2n}{3}+\tilde{o}\big),\dot{\omega}_{y,0}
\big(\frac{2n}{3}+\tilde{o}\big),[\frac{2n}{3}]\big),\big(\gamma_p\big(\frac{2n}{3}\big),\dot{\gamma}_p\big(\frac{2n}{3}\big),[\frac{2n}{3}]\big)\Big)\leq
C_4e^{-\mu \frac{2n}{3}}.
\]
We choose $0\leq d_1<1$ so that
$\big(\gamma_p\big(\frac{2n}{3}+d_1\big),\dot{\gamma}_p\big(\frac{2n}{3}+d_1\big),[\frac{2n}{3}+d_1]\big)=(p,v_p,0)$.
Then from (\ref{3-19}) we obtain

\begin{align}\label{3-21}
d\Big(\big(\omega_{y,0}\big(\frac{2n}{3}+\tilde{o}+d_1\big),\dot{\omega}_{y,0}
\big(\frac{2n}{3}+\tilde{o}+d_1\big),[\frac{2n}{3}+d_1]\big),\big(p,v_p,0\big)\Big)\leq
C_4e^{-\mu \frac{2n}{3}}.
\end{align}
Since $\omega_{y,0}$ is a minimizing curve, then

\begin{align*}
F_{0,\frac{2n}{3}+d_1}\big(y,\omega_{y,0}\big(\frac{2n}{3}+\tilde{o}+d_1\big)\big)
&=F_{\tilde{o},\frac{2n}{3}+\tilde{o}+d_1}\big(y,\omega_{y,0}\big(\frac{2n}{3}+\tilde{o}+d_1\big)\big)\\
&=\int_{\tilde{o}}^{\frac{2n}{3}+\tilde{o}+d_1}L(\omega_{y,0},\dot{\omega}_{y,0},s)ds.
\end{align*}
Let $\eta_1:[0,\frac{2n}{3}+d_1]\to M$ with $\eta_1(0)=y$ and
$\eta_1(\frac{2n}{3}+d_1)=p$ be a Tonelli minimizer such that

\[
F_{0,\frac{2n}{3}+d_1}(y,p)=\int_0^{\frac{2n}{3}+d_1}L(\eta_1,\dot{\eta}_1,s)ds.
\]
Hence, in view of (\ref{3-21}) we have

\begin{align}\label{3-22}
\Big|\int_0^{\frac{2n}{3}+d_1}L(\eta_1,\dot{\eta}_1,s)ds-\int_{\tilde{o}}^{\frac{2n}{3}+\tilde{o}+d_1}L(\omega_{y,0},\dot{\omega}_{y,0},s)ds\Big|\leq
D_{\mathrm{Lip}}C_4e^{-\mu \frac{2n}{3}}.
\end{align}
By (\ref{3-20}) we have $
d\Big(\big(\beta_{x,[\tau]}\big(-\frac{2n}{3}+\tilde{\tau}\big),\dot{\beta}_{x,[\tau]}\big(-\frac{2n}{3}+\tilde{\tau}\big),[-\frac{2n}{3}+\tilde{\tau}]\big)
,\big(\gamma_p\big(-\frac{2n}{3}+[\tau]\big),\dot{\gamma}_p\big(-\frac{2n}{3}+[\tau]\big),[-\frac{2n}{3}+[\tau]]\big)\Big)\leq
C_4e^{-\mu \frac{2n}{3}}. $ Choose $0\leq d_2<1$ so that
$\big(\gamma_p\big(-\frac{2n}{3}+[\tau]-d_2\big),\dot{\gamma}_p\big(-\frac{2n}{3}+[\tau]-d_2\big),[-\frac{2n}{3}+[\tau]-d_2]\big)=(p,v_p,0)$.
From (\ref{3-20}) we have

\begin{align}\label{3-23}
d\Big(\big(\beta_{x,[\tau]}\big(-\frac{2n}{3}+\tilde{\tau}-d_2\big),\dot{\beta}_{x,[\tau]}\big(-\frac{2n}{3}+\tilde{\tau}-d_2\big),[-\frac{2n}{3}+\tilde{\tau}-d_2]\big),
\big(p,v_p,0\big)\Big)\leq C_4e^{-\mu \frac{2n}{3}}.
\end{align}
Since $\beta_{x,[\tau]}$ is a minimizing curve, then

\[
F_{-\frac{2n}{3}+\tilde{\tau}-d_2,\tilde{\tau}}\big(\beta_{x,[\tau]}\big(-\frac{2n}{3}+\tilde{\tau}-d_2\big),x\big)
=\int_{-\frac{2n}{3}+\tilde{\tau}-d_2}^{\tilde{\tau}}L(\beta_{x,[\tau]},\dot{\beta}_{x,[\tau]},s)ds.
\]
Let $\eta_2:[-\frac{2n}{3}+\tilde{\tau}-d_2,\tilde{\tau}]\to M$
with $\eta_2(-\frac{2n}{3}+\tilde{\tau}-d_2)=p$ and
$\eta_2(\tilde{\tau})=x$ be a Tonelli minimizer such that

\[
F_{-\frac{2n}{3}+\tilde{\tau}-d_2,\tilde{\tau}}(p,x)=\int_{-\frac{2n}{3}+\tilde{\tau}-d_2}^{\tilde{\tau}}L(\eta_2,\dot{\eta}_2,s)ds.
\]
Hence, by (\ref{3-23}) we have

\begin{align}\label{3-24}
\Big|\int_{-\frac{2n}{3}+\tilde{\tau}-d_2}^{\tilde{\tau}}L(\eta_2,\dot{\eta}_2,s)ds-\int_{-\frac{2n}{3}+\tilde{\tau}-d_2}^{\tilde{\tau}}L(\beta_{x,[\tau]},\dot{\beta}_{x,[\tau]},s)ds\Big|\leq
D_{\mathrm{Lip}}C_4e^{-\mu \frac{2n}{3}}.
\end{align}

Define a curve
$\tilde{\eta}_2:[\frac{2n}{3}+d_1,\frac{4n}{3}+d_1+d_2]\to M$ by

\[
\tilde{\eta}_2(\varsigma)=\eta_2(\varsigma-\frac{4n}{3}-d_1-d_2+\tilde{\tau})
\]
for $\varsigma\in[\frac{2n}{3}+d_1,\frac{4n}{3}+d_1+d_2]$. Then

\[
\int_{-\frac{2n}{3}+\tilde{\tau}-d_2}^{\tilde{\tau}}L(\eta_2,\dot{\eta}_2,s)ds
=\int_{\frac{2n}{3}+d_1}^{\frac{4n}{3}+d_1+d_2}L(\tilde{\eta}_2,\dot{\tilde{\eta}}_2,\varsigma)d\varsigma.
\]
Note that $\frac{4n}{3}+d_1+d_2-\tau\in\mathbb{N}$ with $n\leq
\frac{4n}{3}+d_1+d_2-\tau\leq2n$. Set
$\bar{k}_n=\frac{4n}{3}+d_1+d_2-\tau$.

Consider the curve $\eta:[0,\tau+\bar{k}_n]\to M$ connecting $y$
and $x$ defined by

\[
\eta(s)= \left\{\begin{array}{ll}
          \eta_1(s),\quad & s\in[0,\frac{2n}{3}+d_1],\\[2mm]
          \tilde{\eta}_2(s),\quad & s\in[\frac{2n}{3}+d_1,\tau+\bar{k}_n].
\end{array}\right.
\]
For $n$ large enough, from (\ref{3-18}) and
(\ref{3-21})--(\ref{3-24}) we have

\begin{align*}
&U^u_n(x,\tau)-\bar{u}(x,[\tau])\\
&\leq u(\eta(0))+\int_0^{\tau+\bar{k}_n}L(\eta,\dot{\eta},s)ds-\bar{u}(x,[\tau])\\
&\leq-h_{0,0}\big(\omega_{y,0}\big(\frac{2n}{3}+\tilde{o}+d_1\big),p\big)
-h_{0,0}\big(p,\beta_{x,[\tau]}\big(-\frac{2n}{3}+\tilde{\tau}-d_2\big)\big)+2D_{\mathrm{Lip}}C_4e^{-\mu\frac{2n}{3}}\\
&\leq 2(C_{\mathrm{Lip}}+D_{\mathrm{Lip}})C_4e^{-\mu\frac{2n}{3}}.
\end{align*}
Note that the above estimate is independent of $(x,\tau)\in
M\times[0,1]$ and $u\in C(M,\mathbb{R}^1)$. Let
$C_5=2(C_{\mathrm{Lip}}+D_{\mathrm{Lip}})C_4$. Then, for $n$ large
enough, we have

\[
U^u_n(x,\tau)-\bar{u}(x,[\tau])\leq C_5e^{-\mu\frac{2n}{3}}, \quad
\forall (x,\tau)\in M\times[0,1].
\]
Hence, there exists a constant $C_6>0$ such that

\begin{align}\label{3-25}
U^u_n(x,\tau)-\bar{u}(x,[\tau])\leq C_6e^{-\mu\frac{2n}{3}}, \quad
\forall n\in\mathbb{N},\ \forall (x,\tau)\in M\times[0,1],
\end{align}
where the constant $C_6$ depends on $u$.

Let $\rho_2=\frac{2}{3}\mu$, $K=\max\{C_3,C_6\}$ and
$\rho=\min\{\rho_1,\rho_2\}$. Then from (\ref{3-17}) and
(\ref{3-25}), we have

\[
\|U^u_n(x,\tau)-\bar{u}(x,[\tau])\|_\infty\leq Ke^{-\rho n}, \quad
\forall n\in\mathbb{N}.
\]
\hfill$\Box$

\section{An example}

In this section we provide an example showing that, even though
the Aubry set of the time-independent Lagrangian system consists
of one hyperbolic periodic orbit, the rate of convergence of
the L-O semigroup cannot be better than $O(\frac{1}{t})$.

\begin{example}\label{ex}
Consider the following Lagrangian

\begin{align*}
L_a:\ T\mathbb{T}^2\to\mathbb{R}^1,\quad
(x,y,\dot{x},\dot{y})\mapsto\frac{1}{2}\dot{x}^2+1-\cos(2\pi
x)+\frac{1}{2}(\dot{y}-c)^2, \ c\in\mathbb{R}^1.
\end{align*}
\end{example}
The associated Hamiltonian $H_a:T^*\mathbb{T}^2\to\mathbb{R}^1$ is
given by
$H_a(x,y,p_1,p_2)=\frac{1}{2}(p_1^2+p_2^2)+cp_2-1+\cos(2\pi x)$.
Observe that
$\Gamma_a=\{0\}\times\mathbb{S}^1\times\{0\}\times\{c\}$ is a
hyperbolic periodic orbit of the Lagrangian system $L_a$. It is
easy to check that the Ma$\mathrm{\tilde{n}}\mathrm{\acute{e}}$
critical value $c(L_a)=0$ and that the probability measure evenly
distributed along $\Gamma_a$---which we shall denote $\nu$---is an
invariant probability measure of the Lagrangian system $L_a$.
Moreover, $\nu$ is the unique minimal measure (0-action minimizing
measure) for $L_a$. Thus, we can conclude that
$\tilde{\mathcal{M}}_0=\Gamma_a=\{0\}\times\mathbb{S}^1\times\{0\}\times\{c\}$.
Note that in this example the Aubry set $\tilde{\mathcal{A}}_0$
coincides with the Mather set $\tilde{\mathcal{M}}_0$, which
implies that $\tilde{\mathcal{A}}_0$ consists of only one
hyperbolic periodic orbit $\Gamma_a$ and
$\mathcal{A}_0=\{0\}\times\mathbb{S}^1$.

\begin{lemma}\label{le4-1}
For each $u\in C(\mathbb{T}^2,\mathbb{R}^1)$ which satisfies
$\min_{\mathcal{A}_0}u=\min_{\mathbb{T}^2}u$, we have
\[
\bar{u}|_{\mathcal{A}_0}\equiv\min_{\mathcal{A}_0}u=\min_{\mathbb{T}^2}u,
\]
where $\bar{u}=\lim_{t\to+\infty}T^a_tu$ is a backward weak KAM
solution of $H_a(x,y,u_x,u_y)=0$.
\end{lemma}

\begin{proof}
For each $(0,y)\in\mathcal{A}_0$, from the definition of $T^a_t$
we have

\[
\bar{u}(0,y)=\lim_{t\to+\infty}T^a_tu(0,y)=\lim_{t\to+\infty}\inf_{(x',y')\in\mathbb{T}^2}
\{u(x',y')+\int_0^tL_a(\gamma,\dot{\gamma})ds\},
\]
where $\gamma:[0,t]\to\mathbb{T}^2$ with $\gamma(0)=(x',y')$ and
$\gamma(t)=(0,y)$ is a Tonelli minimizer. Since $L_a\geq 0$, then
$\bar{u}(0,y)\geq\min_{\mathbb{T}^2}u=\min_{\mathcal{A}_0}u$.

It suffices to show that
$\bar{u}(0,y)\leq\min_{\mathbb{T}^2}u=\min_{\mathcal{A}_0}u$. Take
$(0,y_*)\in\mathcal{A}_0$ with $u(0,y_*)=\min_{\mathcal{A}_0}u$.
Let $\tilde{y}_*\in\mathbb{R}^1$ be an arbitrary point in the
fiber over $y_*$. For $t>0$, consider the following two curves

\[
\tilde{\gamma}_{2,c}:[0,t]\to\mathbb{R}^1,\quad s\mapsto
cs+\tilde{y}_*,
\]
and

\[
\tilde{\gamma}_{2,c'}:[0,t]\to\mathbb{R}^1,\quad s\mapsto
c's+\tilde{y}_*,\ c'\in\mathbb{R}^1,
\]
with $\tilde{\gamma}_{2,c'}(t)=\tilde{y}$, where $\tilde{y}$ is a
point in the fiber over $y$ such that $\tilde{y}$ and
$\tilde{\gamma}_{2,c}(t)$ are in the same fundamental domain in
$\mathbb{R}^1$. Let
$r=\tilde{\gamma}_{2,c'}(t)-\tilde{\gamma}_{2,c}(t)=(c'-c)t$. Then
$|r|\leq 1$ and $|c'-c|\leq\frac{1}{t}$. Let
$\tilde{\gamma}_{c'}=(0,\tilde{\gamma}_{2,c'}):[0,t]\to\mathbb{R}^2$
and $\gamma_{c'}=\pi\tilde{\gamma}_{c'}$, where
$\pi:\mathbb{R}^2\to\mathbb{T}^2$ is the standard universal
covering projection. Then $\gamma_{c'}:[0,t]\to\mathbb{T}^2$ is a
curve connecting $(0,y_*)$ and $(0,y)$. Hence, we have

\[
T^a_tu(0,y)\leq
u(\gamma_{c'}(0))+\int_0^tL(\gamma_{c'},\dot{\gamma}_{c'})ds=u(0,y_*)+\frac{1}{2}\int_0^t(c'-c)^2ds\leq
u(0,y_*)+\frac{1}{2t}.
\]
Let $t\to+\infty$, to deduce

\[
\bar{u}(0,y)=\lim_{t\to+\infty}T^a_tu(0,y)\leq
u(0,y_*)=\min_{\mathcal{A}_0}u=\min_{\mathbb{T}^2}u.
\]
\end{proof}

In the following we show that there exist $u\in
C(\mathbb{T}^2,\mathbb{R}^1)$, $(x_0,y_0)\in\mathbb{T}^2$ and
$\{t_n\}_{n=1}^{+\infty}$ with $t_n\to+\infty$ as $n\to+\infty$
such that

\[
|T^a_{t_n}u(x_0,y_0)-\bar{u}(x_0,y_0)|\geq O(\frac{1}{t_n}),\quad
n\to+\infty.
\]

Set $(x_0,y_0)=(0,\frac{1}{2})\in\mathbb{T}^2$. Let
$(\tilde{x}_0,\tilde{y}_0)\in\mathbb{R}^2$ denote a generic point
in the fiber over $(x_0,y_0)$, i.e.,
$\pi(\tilde{x}_0,\tilde{y}_0)=(x_0,y_0)$. Define a continuous
function on $\mathbb{R}^2$ as follows: for each
$(\tilde{x},\tilde{y})\in\mathbb{R}^2$

\[
\tilde{u}(\tilde{x},\tilde{y})=\left\{
                               \begin{array}{ll}
                                  \delta-|\tilde{y}-\tilde{y}_0|, &
                                  |\tilde{y}-\tilde{y}_0|\leq\delta,\\
                                  0, & \mathrm{otherwise},
                         \end{array}
                         \right.
\]
where $0<\delta<\frac{1}{2}$. Then we can define a continuous
function on $\mathbb{T}^2$ as
$u(x,y)=\tilde{u}(\tilde{x},\tilde{y})$ for all
$(x,y)\in\mathbb{T}^2$, where $(\tilde{x},\tilde{y})$ is an
arbitrary point in the fiber over $(x,y)$. From Lemma \ref{le4-1},
we have
$\bar{u}(0,y_0)=\min_{\mathbb{T}^2}u=\min_{\mathcal{A}_0}u=0$.

Now fix a point $(0,\tilde{y}^0_0)\in\mathbb{R}^2$ in the fiber
over $(0,y_0)$. Then there exist
$\{(0,\tilde{y}^n_0)\}_{n=1}^{+\infty}$ in the fiber over
$(0,y_0)$ and $\{t_n\}_{n=1}^{+\infty}$ with $t_n\to+\infty$ as
$n\to+\infty$ such that
$|\tilde{y}^n_0-ct_n-\tilde{y}^0_0|\leq\frac{\delta}{2}$. Let
$\tilde{z}^n=\tilde{y}^n_0-ct_n$, $\forall n$. Then
$|\tilde{z}^n-\tilde{y}^0_0|\leq\frac{\delta}{2}$, $\forall n$.
For each $t_n$ there is $(x_n,\xi_n)\in\mathbb{T}^2$ such that

\begin{align}\label{4-1}
T^a_{t_n}u(0,y_0)=u(x_n,\xi_n)+\int_0^{t_n}L_a(\gamma_n,\dot{\gamma}_n)ds,
\end{align}
where
$\gamma_n=(\gamma_{1,n},\gamma_{2,n}):[0,t_n]\to\mathbb{T}^2$ with
$\gamma_n(0)=(x_n,\xi_n)$ and $\gamma_n(t_n)=(0,y_0)$ is a Tonelli
minimizer. We assert that $x_n=0$, $\forall n$, i.e.,
$(x_n,\xi_n)\in\mathcal{A}_0$, $\forall n$. For, otherwise, there
would be $x_n\neq 0$ for some $n$. Then we have

\begin{align}\label{4-2}
u(0,\xi_n)+\frac{1}{2}\int_0^{t_n}(\dot{\gamma}_{2,n}-c)^2ds<
u(x_n,\xi_n)+\int_0^{t_n}\big(\frac{1}{2}\dot{\gamma}_{1,n}^2+1-\cos(2\pi\gamma_{1,n})
+\frac{1}{2}(\dot{\gamma}_{2,n}-c)^2\big)ds.
\end{align}
Let $\gamma'_n=(0,\gamma_{2,n}):[0,t_n]\to\mathbb{T}^2$. Then $\gamma'_n$ is a curve in
$\mathbb{T}^2$ connecting $(0,\xi_n)$ and $(0,y_0)$. In view of (\ref{4-2}), we have

\begin{align*}
T^a_{t_n}u(0,y_0) & \leq
u(0,\xi_n)+\frac{1}{2}\int_0^{t_n}(\dot{\gamma}_{2,n}-c)^2ds\\
&<u(x_n,\xi_n)+\int_0^{t_n}\big(\frac{1}{2}\dot{\gamma}_{1,n}^2+1-\cos(2\pi\gamma_{1,n})
+\frac{1}{2}(\dot{\gamma}_{2,n}-c)^2\big)ds,
\end{align*}
which contradicts (\ref{4-1}). Hence $x_n=0$, $\forall n$. It is easy to see that

\begin{align}\label{4-3}
T^a_{t_n}u(0,y_0)=u(0,\xi_n)+\frac{1}{2}\int_0^{t_n}(\dot{\gamma}_{2,n}-c)^2ds,\quad
n=1,2,\cdots.
\end{align}

In view of the lifting property of the covering projection, for
each $n$ there is a unique curve
$\tilde{\gamma}_{2,n}:[0,t_n]\to\mathbb{R}^1$ with
$\pi\tilde{\gamma}_{2,n}=\gamma_{2,n}$ and
$\tilde{\gamma}_{2,n}(t_n)=\tilde{y}^n_0$. Set
$\tilde{\xi}^n=\tilde{\gamma}_{2,n}(0)$. Then
$\pi\tilde{\xi}^n=\xi_n$. Moreover, $\tilde{\gamma}_{2,n}$ has the
form $\tilde{\gamma}_{2,n}(s)=c_ns+\tilde{\xi}^n$, $s\in[0,t_n]$,
$c_n\in\mathbb{R}^1$. It is clear that
$\tilde{\xi}^n=\tilde{y}^n_0-c_nt_n$.

If $|\tilde{\xi}^n-\tilde{z}^n|\leq\frac{\delta}{4}$, then from
$|\tilde{z}^n-\tilde{y}^0_0|\leq\frac{\delta}{2}$, we have
$|\tilde{\xi}^n-\tilde{y}^0_0|\leq\frac{3\delta}{4}$. Therefore,
in view of (\ref{4-3}),

\begin{align}\label{4-4}
T^a_{t_n}u(0,y_0)=u(0,\xi_n)+\frac{1}{2}\int_0^{t_n}(\dot{\gamma}_{2,n}-c)^2ds\geq
u(0,\xi_n)=\tilde{u}(0,\tilde{\xi}^n)\geq \frac{\delta}{4}.
\end{align}
By (\ref{4-4}) we may deduce that there can be only a finite
number of $\tilde{\xi}^n$'s such that
$|\tilde{\xi}^n-\tilde{z}^n|\leq\frac{\delta}{4}$. Suppose not.
There are $\{t_{n_i}\}_{i=1}^{+\infty}$ and
$\{\tilde{\xi}^{n_i}\}_{i=1}^{+\infty}$ such that
$T^a_{t_{n_i}}u(0,y_0)\geq \frac{\delta}{4}$, $i=1,2,\cdots$,
which contradicts
$\lim_{i\to+\infty}T^a_{t_{n_i}}u(0,y_0)=\bar{u}(0,y_0)=0$.

For $\tilde{\xi}^n$ with
$|\tilde{\xi}^n-\tilde{z}^n|>\frac{\delta}{4}$, we have

\[
\frac{\delta}{4}<|\tilde{\xi}^n-\tilde{z}^n|=|(c-c_n)t_n|,
\]
which implies that $|c-c_n|>\frac{\delta}{4t_n}$. Then

\[
T^a_{t_n}u(0,y_0)=u(0,\xi_n)+\frac{1}{2}\int_0^{t_n}(\dot{\gamma}_{2,n}-c)^2ds>\frac{\delta^2}{32t_n}.
\]
Therefore,

\[
|T^a_{t_n}u(0,y_0)-\bar{u}(0,y_0)|=|T^a_{t_n}u(0,y_0)|>\frac{\delta^2}{32t_n},
\]
i.e.,

\[
|T^a_{t_n}u(0,y_0)-\bar{u}(0,y_0)|\geq O(\frac{1}{t_n}), \quad
n\to+\infty.
\]

\vskip1cm



\end{document}